\documentclass[leqno]{amsart}

\usepackage{amssymb}
\usepackage{amsmath}

\usepackage{hyperref}

\makeatletter
\def\thethm{\thesection.\@arabic\c@thm}
\def\thelem{\thesection.\@arabic\c@thm}
\def\thecor{\thesection.\@arabic\c@thm}
\def\theprop{\thesection.\@arabic\c@thm}
\def\therem{\thesection.\@arabic\c@thm}
\makeatother

\newtheorem{thm}{Theorem}[section]

\newtheorem{lem}[thm]{Lemma}

\theoremstyle{definition}

\newtheorem{rem}[thm]{Remark}

\numberwithin{equation}{section}

\def\ds{\displaystyle}

\newcommand{\R}{{\mathbb{R}}}

\def\ol{\overline}
\def\sm{\setminus}

\def\bA{{\mathbf{A}}}
\def\bb{{\mathbf{b}}}
\def\bu{{\mathbf{u}}}
\def\bv{{\mathbf{v}}}
\def\bz{{\mathbf{z}}}
\def\bxi{{\mathbf{\Xi}}}

\def\m{{\mathfrak{m}}}

\def\e{{\mathrm{e}}}

\begin{document}

\title[Boundedness of Solutions to Coercive Systems]{Boundedness of Solutions to a Class of Coercive Systems with Morrey Data}

\author{Dian K. Palagachev}
\address{DMMM, Politecnico di Bari, 70 125 Bari, Italy}
\email{dian.palagachev@poliba.it}  

\author{Lubomira G. Softova}
\address{Department of Mathematics, University of Salerno, Salerno, Italy}
\email{lsoftova@unisa.it}  
   

\date{\today}


\begin{abstract}
We prove global essential boundedness of the weak solutions $\bu\in W^{1,p}_0(\Omega;\mathbb{R}^N)$ to the quasilinear system
$$
\mathrm{div\,}\big(\bA(x,\bu,D\bu)\big)= \bb(x,\bu,D\bu).
$$
The principal part $\bA(x,\bu,D\bu)$ of the differential operator is componentwise coercive and supports controlled growths with respect to $\bu$ and $D\bu,$ while the lower order term $\bb(x,\bu,D\bu)$ exhibits componentwise controlled gradient growth. The $x$-behaviour of the nonlinearities is governed in terms of Morrey spaces.
\end{abstract}

\maketitle

\section{Introduction}\label{sec1}

Solutions to various real world problems realize minimal energy of suitable nonlinear functionals, and the central problem of the Calculus of Variations is to get existence of such solutions and to study their qualitative properties such as multiplicity, symmetry, monotonicity, etc. In all these issues, it is the machinery of the nonlinear functional analysis which plays a crucial role. On the other hand, each minimizer of a variational
functional solves weakly the corresponding Euler--Lagrange equation/system and this fact
allows to employ the powerful theory of PDEs as an additional tool. The Euler--Lagrange equations are divergence form PDEs, usually elliptic
and nonlinear, and their weak solutions (the minimizers) possess some basic
minimal smoothness. The regularity theory of general (non necessary variational)
divergence form elliptic PDEs establishes how the smoothness of the data reflects 
on the regularity of the solution, already obtained under very general circumstances.
Once having better smoothness, powerful tools of functional analysis apply to infer more precise properties of the solution. The importance
of these issues becomes more evident if dealing with variational problems for discontinuous functionals and over domains with non-smooth boundaries when many of the classical nonlinear analysis techniques fail.

This paper deals with boundedness properties of the weak solutions to divergence form nonlinear systems with data belonging to suitable Morrey spaces. More precisely, we consider the
Dirichlet problem  
\begin{equation}\label{P}
\begin{cases}
\bu\in W^{1,p}_0(\Omega;\mathbb{R}^N),\quad N>1\\
\mathrm{div\,}\big(\bA(x,\bu,D\bu)\big)= \bb(x,\bu,D\bu)\quad   \text{weakly in}\ \Omega\\
\end{cases}
\end{equation}
over a bounded domain $\Omega\subset\R^n,$ $n\geq2,$  with generally low regular boundary $\partial\Omega$ and where the nonlinear terms are given by the Carath\'eodory maps 
$\bA\colon \Omega\times\R^N\times \R^{N\times n}\to\R^{N\times n},$ $\bA=\big\{a^\alpha_i\big\}_{i=1,\ldots,n}^{\alpha=1,\ldots,N},$
$\bb\colon \Omega\times\R^N\times \R^{N\times n}\to\R^{N},$ $\bb=(b^1,\ldots,b^N).$ Let us stress the reader attention at the very beginning that, even if \eqref{P} can be viewed as the Euler--Lagrange equation for many important functionals from mathematical
physics and differential geometry  (such as the nonlinear Schr\"odinger,
Gunzburg--Landau, harmonic maps between Riemannian manifolds, Gross--Pitaevskii, etc.), the structure of the nonlinear operator here considered in not necessary variational.

Our main goal is to obtain sufficient conditions  ensuring that \textit{any} weak solution
$\bu(x)=\big(u^1(x),\ldots,u^N(x)\big)\colon \Omega\to \R^N$ of \eqref{P} is \textit{globally essentially bounded} in $\Omega$ when $1<p\leq N$ (the boundedness is an immediate consequence of the Sobolev imbedding theorem when $p> N$). These conditions require relevant coercivity of the differential operator in \eqref{P}, suitable growths of the nonlinearities with respect to the solution $\bu$ and its gradient $D\bu,$ while the $x$-behaviour of $\bA$ and $\bb$  will be controlled in terms of Morrey functional scales (see also \cite{Sf} for particular systems of elliptic type). 

The boundedness problem for the weak solutions to elliptic \textit{equations} $(N=1)$ is  completely solved. The seminal result of De Giorgi \cite{DG1} and Nash \cite{N} ensures boundedness and H\"older continuity of the $W^{1,2}_0$-weak solutions to \textit{linear} elliptic equations with $L^p$-coefficients and it was later extended by Ladyzhenskaya and Ural'tseva \cite{LU} to the case of \textit{quasilinear} equations. Recently, boundedness and H\"older continuity of the weak solutions to general quasilinear equations have been proved (\cite{BP-IUMJ,BP-CalcVar} when $p=2$ and \cite{BPS,BPS-arxiv} when $p\in(1,n]$) allowing control in terms of Morrey spaces for the $x$-behaviour of the nonlinear terms.

The situation changes drastically when passing to systems $(N>1).$ The boundedness properties of the weak solutions to \eqref{P} are strongly conditioned by the De~Giorgi example \cite{DG2} of \textit{linear elliptic system} with \textit{unbounded} weak solution where the lack of boundedness is due to the particular structure of the off-diagonal elements of the coefficients matrix. For linear systems with ``almost'' diagonal structure Ladyzhenskaya and Ural'tseva  proved in \cite{LU} local boundedness of the weak solutions (see also \cite{Lan}), while Ne\v{c}as and Star\'a \cite{NS} derived a maximum principle for quasilinear systems which are diagonal for large values of $\bu.$ 
Boundedness of the weak solutions to nondiagonal systems has been obtained by Meier in \cite{M}
assuming that the operator in \eqref{P} is   \textit{coercive}  
$$
\sum_{\alpha=1}^N \bA^\alpha(x,\bu,D\bu)\cdot D\bu^\alpha\geq \varkappa |D\bu|^p - \text{lower order terms}\,(x,\bu)
$$
and that the indicator function
$$
\sum_{\alpha,\beta=1}^N \dfrac{\bu^\alpha\bu^\beta}{|\bu|^2} \bA^\alpha(x,\bu,D\bu)\cdot D\bu^\beta
$$
is nonnegative on the set where $|\bu|$ is large. Apart from the fact that this last condition depends on the particular solution, it is very difficult to check it in general.

In \cite{Bjorn} Bj\"orn considered quasilinear systems which are not too far from being diagonal. Precisely, the author required  \textit{componentwise coercivity} of the differential operator in \eqref{P},
$$ 
\bA^\alpha(x,\bu,D\bu)\cdot D\bu^\alpha\geq \varkappa |Du^\alpha|^p - \text{lower order terms}\,(x,\bu)
$$
for any $\alpha\in\{1,\ldots,N\}$ which means that the $\alpha$-th equation of the system \eqref{P} is \textit{coercive} with respect to the gradient of the $\alpha$-th component of the weak solution. Assuming additionally $(p-1)$-growths of the nonlinear terms
$$
\begin{cases}
\hfill
\big|\bA(x,\bu,D\bu)\big|= &\!\!\!\! \mathcal{O}\left(\varphi(x)+|\bu|^{p-1}+|D\bu|^{p-1}\right),\\
\hfill
\big|\bb(x,\bu,D\bu)\big|= &\!\!\!\! \mathcal{O}\left(\psi(x)+|\bu|^{p-1}+|D\bu|^{p-1}\right)
\end{cases}
$$
with $\varphi$ and $\psi$ taken in suitable Lebesgue spaces, Bj\"orn proved local boundedness and almost everywhere classical differentiability for the $W^{1,p}$-weak solutions when $1<p\leq2.$ (Actually, the hypotheses in \cite{Bjorn} are expressed in a more general form allowing to treat degenerate operators with $p$-admissible weights and also variational inequalities.)

Recently maximum principle results have been obtained for componentwise coercive systems
with lower-order term $\bb(x,\bu,D\bu)$ satisfying a sort of sign-condition with respect to $\bu,$ and $\varphi,$ $\psi$ belonging to suitable Lebesgue (\cite{Leo}) or Morrey (\cite{Sof,Sof2}) spaces.

\bigskip

In the present paper we consider \textit{componentwise coercive} systems \eqref{P}. The nonlinear terms are subject of \textit{controlled growth assumptions} that are the optimal ones giving sense of the concept of weak solution. Precisely, the principal part of the differential operator is supposed to satisfy
$$
\big|\bA(x,\bu,D\bu)\big|=  \mathcal{O}\left(\varphi(x)+|\bu|^{\frac{p^*(p-1)}{p}}+|D\bu|^{p-1}\right)
$$
with the Sobolev conjugate $p^*$ of $p,$ while on the lower order term we require \textit{gradient componentwise controlled growth}
$$
\big|b^\alpha(x,\bu,D\bu)\big|= \mathcal{O}\left(\psi(x)+|\bu|^{p^*-1}+|Du^\alpha|^{\frac{p(p^*-1)}{p^*}}\right)\quad \forall \alpha\in\{1,\ldots,N\}.
$$
Actually, the last condition means that the right-hand side of the $\alpha$-th equation grows as the $\frac{p(p^*-1)}{p^*}$-th power of the gradient of the $\alpha$-th solution component. Indeed, this hypothesis is more restrictive than the general controlled growth assumption when $\big|b^\alpha(x,\bu,D\bu)\big|$ can grow as $|D\bu|^{\frac{p(p^*-1)}{p^*}}$ and we do not know whether this is only a technical restriction or it is intrinsically related to the nature of the systems studied. The functions $\varphi$ and $\psi$ are taken in the Morrey spaces $L^{r,\lambda}$ and $L^{s,\mu},$ respectively, with $(p-1)r+\lambda>n$ and $ps+\mu>n,$ and the particular situation when $\lambda=\mu=0$ covers also the case of $L^r/L^s$ data. 

The main result of the paper (Theorem~\ref{thm}) asserts global boundedness of any weak solution to the problem \eqref{P}. Our technique is inspired by that already used in \cite{BPS-arxiv} to get  H\"older continuity for weak solutions to coercive \textit{equations} with Morrey data. It relies on exact  decay estimates for the total mass of each component $u^\alpha$ of the weak solution taken over the level sets of $u^\alpha.$ However, the presence of Morrey data $\varphi$ and $\psi,$ and the specificity of the controlled growth assumptions require this mass to be taken with respect to a positive Radon
measure $\m,$ depending on $\varphi$ and $\psi,$ but also on a suitable power of the weak solution itself.
The Morrey integrability of $\varphi$ and $\psi$ allows to employ precise Sobolev inequalities of trace type proved by D.R.~Adams \cite{Ad} and V.G.~Maz'ya \cite{Mz1,Mz2} in order to estimate the $\m$-mass of $u^\alpha$ in terms of the $p$-energy of $u^\alpha$ for each $\alpha.$  To manage the nonlinear part of $\m$ that depends on the solution $\bu$ we rely also on the higher gradient integrability in the the spirit of  Gehring and Giaquinta. This, combined with the controlled growth conditions, gives an estimate for the $p$-energy of $u^\alpha$ in terms of small multiplier of the same quantity plus a suitable power of the $u^\alpha$-level set $\m$-measure. It is at that point that the
$|Du^\alpha|^{\frac{p(p^*-1)}{p^*}}$-growth of $|b^\alpha(x,\bu,D\bu)\big|$ plays a crucial role since, on the level set of $u^\alpha$ we 
can control only $Du^\alpha$ but not  $Du^\beta$
for $\beta\neq\alpha.$ It is worth noting that,
allowing  \textit{full gradient} growth
$|D\bu|^{\frac{p(p^*-1)}{p^*}}$ of $|\bb(x,\bu,D\bu)\big|,$ the technique of Bj\"orn from \cite{Bjorn} would give the boundedness result only for very small values of $p$ in the range $\left(1,\frac{2n}{n+1}\right).$ Once having a  good decay estimate for the solution total $\m$-mass, the global boundedness follows by a maximum principle result due to Hartman and Stampacchia.

As consequence of the Morrey control with respect to $x$ of the nonlinear terms in \eqref{P} we obtain (Theorem~\ref{thm2}) also Morrey regularity for  the gradient of the \textit{bounded} weak solutions.

\subsection*{Acknowledgments.}
The authors are members of INdAM/GNAMPA.

\section{Hypotheses and main results}\label{sec2}

We will use throughout the paper two types of indices: the subscript $i$ indicates the $i$-th component of a point $x\in\R^n$ and varies from $1$ to $n,$  while the superscript $\alpha$ runs from $1$ to $N$ and stands for the $\alpha$-th component of an $N$-dimensional vector.  The boldface small roman letters $\bb,\bu,\bv,\ldots$ denote $N$-dimensional vector-valued functions whereas boldface capital letters $\bA,\bxi$ stand for $N\times n$ matrices. The $n$-dimensional ball centered at $x$ and of radius $\rho$ will be denoted by  $B_\rho(x)$ and $|E|$ stands for the Lebesgue measure of a measurable set $E\subseteq\Omega.$

Given $p\in(1,\infty),$ the Sobolev space  of once weakly differentiable functions $u\colon\ \Omega\to\R$ belonging to $L^p(\Omega)$ together with the gradient $Du$ is denoted as usual by $W^{1,p}(\Omega),$ the norm in $W^{1,p}(\Omega)$ is given by 
$$
\|u\|_{W^{1,p}(\Omega)}:= \|u\|_{L^{p}(\Omega)}+\|Du\|_{L^{p}(\Omega)},
$$
and $W^{1,p}_0(\Omega)$ stands for the closure of $C^{\infty}_0(\Omega)$ with respect to that norm. Further on, $W^{1,p}(\Omega;\R^N)$ is the collection of all vector-valued functions $\bu\colon \Omega\to\R^N,$ $\bu(x)=\big(u^1(x),\ldots,u^N(x)\big),$ such that $u^\alpha\in W^{1,p}(\Omega)$ and
$$
\|\bu\|_{W^{1,p}(\Omega;\R^N)}:=\sum_{\alpha=1}^N \|u^\alpha\|_{W^{1,p}(\Omega)}.
$$
We will denote by $p^*$ the \textit{Sobolev conjugate} of $p,$ that is,
$$
p^* =
\begin{cases}
\ds
\frac{np}{n-p} & \text{if}\ p<n,\\[6pt]
\text{arbitrary large number}>n & \text{if}\ p\geq n.
\end{cases}
$$

Let us recall, for reader's convenience, the definition of the Morrey spaces. Given a bounded domain $\Omega\subset\mathbb{R}^n,$ $r\in[1,\infty)$ and $\lambda\in[0,n],$ a function $u\in L^r(\Omega)$ belongs to the Morrey class $L^{r,\lambda}(\Omega)$ if
$$
\|u\|_{L^{r,\lambda}(\Omega)}:=\sup_{\underset{\rho\in(0,\mathrm{diam\,}\Omega)}{x_0\in\Omega,}}
\left( \dfrac{1}{\rho^{\lambda}}
\int_{B_\rho(x_0)\cap\Omega} |u(x)|^r \; dx\right)^{1/r}<\infty.
$$
This quantity defines a norm which makes $L^{r,\lambda}(\Omega)$ a Banach space. The limit cases $\lambda=0$ and $\lambda=n$ give rise, respectively, to $L^r(\Omega)$ and $L^\infty(\Omega).$ It is worth noting (see \cite{Pic}) that the imbedding
$$
L^{r_1,\lambda_1}(\Omega) \subseteq L^{r_2,\lambda_2}(\Omega)
$$
holds \textit{if and only if}
$$
r_1\geq r_2\quad \text{and}\quad \dfrac{r_1}{n-\lambda_1} \geq \dfrac{r_2}{n-\lambda_2}.
$$

Throughout the paper we will consider a bounded domain $\Omega\subset\R^n$ with $n\geq2,$ and will suppose that the boundary $\partial\Omega$ satisfies a \textit{measure density condition} which is a two-sided version of the the so-called  \textit{(A)-property} of Ladyzhenskaya and  Ural'tseva (see \cite{LU}) which requires  that for each $x\in\ol\Omega$ the Lebesgue measure of $B_\rho(x)\cap\Omega$ is comparable to the measure of the ball $B_\rho(x)$ itself.  Precisely, we suppose that there exists a constant $A_\Omega\in(0,1)$ such that
\begin{equation}\label{A}
A_\Omega \rho^n\leq
\big|B_\rho(x)\cap\Omega\big|\leq \big(1-A_\Omega\big) \rho^n\quad \forall x\in\ol\Omega,\ \forall \rho\in(0,\mathrm{diam\,}\Omega).
\end{equation}
The lower bound above excludes interior cusps at each point of the boundary and this ensures the validity of the Sobolev imbedding theorem within the spaces $W^{1,p}(\Omega).$ The upper bound instead excludes exterior cusps at $\partial\Omega$ and this serves, as will be seen in Lemma~\ref{lem4} below, to obtain higher gradient integrability for the weak solutions of \eqref{P}. The (A)-property holds for example when $\partial\Omega$ supports the \textit{uniform} interior and exterior cone conditions. In particular,  \eqref{A} is always verified if $\partial\Omega$ is $C^1$-smooth, or Lipschitz, or Reifenberg flat (cf. \cite{NA} and the references therein).

It is worth noting that the results here presented remain valid also in less ``regular'' domains $\Omega$ when the measure density condition \eqref{A} is replaced by the more general one expressed in terms of \textit{variational $p$-capacity} that requires the complement $\mathbb{R}^n\setminus\Omega$ to be \textit{uniformly $p$-thick} (see \cite{BPS-arxiv} for more details).

\smallskip

Turning back to \eqref{P}, recall that a function $\bu\in W^{1,p}_0(\Omega;\R^N)$  is called a \textit{weak solution\/} of the problem
$\eqref{P}$ if
\begin{align}\label{WS}
&\int_\Omega 
\sum_{\alpha=1}^N \sum_{i=1}^n
a_i^\alpha\big(x,\bu(x),D\bu(x)\big) D_iv^\alpha(x)\;dx\\
\nonumber
&\qquad\qquad
+\int_\Omega \sum_{\alpha=1}^N b^\alpha\big(x,\bu(x),D\bu(x)\big)v^\alpha(x)\;dx=0
\end{align}
for each test function $\bv\in W^{1,p}_0(\Omega;\R^N).$

Indeed, the concept of weak solution to \eqref{P} makes sense only if the integrals involved in \eqref{WS} are convergent and this is ensured by imposing suitable growth requirements on the nonlinear terms. The optimal conditions of this kind are known as \textit{controlled growth conditions} and have the form
\begin{equation}\label{CG}
\begin{cases}
\hfill
\big|a^\alpha_i(x,\bz,\bxi)\big|\leq &\!\!\!\! \Lambda\Big(\varphi(x)+|\bz|^{\frac{p^*(p-1)}{p}}+|\bxi|^{p-1}\Big),\\[6pt]
\hfill
\big|b^\alpha(x,\bz,\bxi)\big|\leq &\!\!\!\! \Lambda\Big(\psi(x)+|\bz|^{p^*-1} +|\bxi|^{\frac{p(p^*-1)}{p^*}}\Big)
\end{cases}
\end{equation}
for all $i\in\{1,\ldots,n\},$ all $\alpha\in\{1,\ldots,N\},$ almost all $x\in\Omega,$ all $\bz\in\R^N$ and all $\bxi=\{\xi^\alpha_i\}\in\R^{N\times n},$  where $\Lambda$ is a positive constant, and $\varphi,\psi$ are nonnegative functions satisfying
$$
\varphi\in L^r(\Omega),\ r\geq \dfrac{p}{p-1};\quad
\psi \in L^s(\Omega),\ s\geq \dfrac{np}{np+p-n}.
$$

Actually, the conditions \eqref{CG} are away of being sufficient to ensure essential boundedness of the weak solution to \eqref{P} when $p\leq n,$ and we will straighten these, complementing them by a sort of coercivity assumption of the operator, together with Morrey integrability of the functions $\varphi$ and $\psi:$
\begin{equation}\label{M}
\left\{
\begin{array}{llll}
\ds
\varphi\in L^{r,\lambda}(\Omega), & r>\dfrac{p}{p-1}, & \lambda\in[0,n), & (p-1)r+\lambda>n,\\[6pt]
\ds
\psi\in L^{s,\mu}(\Omega), & s>\dfrac{np}{np+p-n}, & \mu\in[0,n), & ps+\mu>n.
\end{array}
\right.
\end{equation}

We will suppose that the $\alpha$-th equation of the system \eqref{P} is coercive with respect to the gradient of the $\alpha$-th component of the solution, that is,

\smallskip

$\bullet$ \textit{Componentwise coercivity of the differential operator:} There exist  positive constants $\varkappa$ and $\Lambda$ such that for each $\alpha\in\{1,\ldots,N\}$ one has
\begin{equation}\label{CWC}
\begin{cases}
\bA^\alpha(x,\bz,\bxi)\cdot \bxi^\alpha:=&\!\!\!\ds \sum_{i=1}^n a^\alpha_i(x,\bz,\bxi)\xi^\alpha_i
\\[6pt]
\hfill \geq &\!\!\!\varkappa |\bxi^\alpha|^p-\Lambda|\bz|^{p^*} -\Lambda \varphi^{\frac{p}{p-1}}(x)
\end{cases}
\end{equation}
for almost all $x\in\Omega,$ all $\bz=(z^1,\ldots,z^N)\in \R^N$ and all $\bxi=(\bxi^1,\ldots,\bxi^N)^{\mathrm{T}}\in \R^{N\times n}$ with $\varphi$  as in \eqref{M}.

\smallskip

$\bullet$
We assume further \textit{controlled growths} of the principal part $\bA(x,\bu,D\bu)$ (that is, $\eqref{CG}_1$) and \textit{gradient componentwise controlled growths} of the lower order term $\bb(x,\bu,D\bu):$
\begin{equation}\label{CWCG}
\begin{cases}
\hfill
\big|a^\alpha_i(x,\bz,\bxi)\big|\leq &\!\!\!\! \Lambda\Big(\varphi(x)+|\bz|^{\frac{p^*(p-1)}{p}}+|\bxi|^{p-1}\Big),\\[6pt]
\hfill
\big|b^\alpha(x,\bz,\bxi)\big|\leq &\!\!\!\! \Lambda\Big(\psi(x)+|\bz|^{p^*-1} +|\bxi^\alpha|^{\frac{p(p^*-1)}{p^*}}\Big)
\end{cases}
\end{equation}
for all $i\in\{1,\ldots,n\},$ all $\alpha\in\{1,\ldots,N\},$ almost all $x\in\Omega,$ all $\bz=(z^1,\ldots,z^N)\in \R^N$ and all $\bxi=(\bxi^1,\ldots,\bxi^N)^{\mathrm{T}}\in \R^{N\times n}$ with $\varphi$ and $\psi$  as in \eqref{M}.

The meaning of $\eqref{CWCG}_2$ is that the lower order term of the $\alpha$-th equation  supports $\frac{p(p^*-1)}{p^*}$-growth with respect to the gradient $Du^\alpha$ of the $\alpha$-th component of the solution $\bu.$ 
Indeed, $\eqref{CWCG}_2$ is more restrictive than $\eqref{CG}_2$ but it anyway allows to consider systems with general enough lower order terms. It will be clear from the proofs below, that the boundedness result (Theorem~\ref{thm}) remains valid if substitute $\eqref{CWCG}$ with $\eqref{CG}$ and require additionally the \textit{sign condition}
\begin{equation}\label{SC}
b^\alpha(x,\bz,\bxi){.}\mathrm{sign\,}z^\alpha \geq 
- \Lambda\Big(\psi(x)+|\bz|^{p^*-1} +|\bxi^\alpha|^{\frac{p(p^*-1)}{p^*}}\Big).
\end{equation}
In this sense, our result generalizes these proved in \cite{Leo} and \cite{Sof,Sof2} where \eqref{CG} are complemented with the requirement that for each $\alpha\in\{1,\ldots,N\}$ one has
$$
b^\alpha(x,\bz,\bxi){.}\mathrm{sign\,}z^\alpha \geq 0.
$$

Throughout the article the letter $C$ will denote various positive constants
depending on known quantities appearing in the above conditions, with the  omnibus term 
\textit{``known quantities''}
 including  $n,$ $N,$ $p,$ $\varkappa,$ $\Lambda,$ $r,$ $\lambda,$ $s,$ $\mu,$ $\|\varphi\|_{L^{r,\lambda}(\Omega)},$ $\|\psi\|_{L^{s,\mu}(\Omega)},$ $|\Omega|$ and  $A_\Omega.$

Our main result asserts essential boundedness  of the $W^{1,p}_0(\Omega;\R^N)$-weak solutions to the Dirichlet problem \eqref{P}. Namely,
\begin{thm}\label{thm}
Let $p\in (1,n]$ and assume \eqref{A}, \eqref{M} and \eqref{CWC}. Suppose moreover that
either \eqref{CWCG} or \eqref{CG} with \eqref{SC} are verified. Then any weak $W^{1,p}_0(\Omega;\R^N)$-solution of \eqref{P} is essentially bounded.  Precisely, there exists a constant $M$ depending on known quantities, on
$\|D\bu\|_{L^{p}(\Omega;\R^{N\times n})}$ and on the uniform integrability of $|D\bu|^p$ (cf. Remark~\ref{rem-HI} below) such that
\begin{equation}\label{EST}
\|\bu\|_{L^\infty(\Omega;\R^N)}\leq M.
\end{equation}
\end{thm}

It is worth noting that the  requirements \eqref{M}, as well as the growth assumptions \eqref{CWCG}, respectively \eqref{CG}, are \textit{sharp} in order to have essential boundedness of the weak solutions to \eqref{P}.
This follows easily on the base of the examples given in \cite[Section~4.2]{BPS-arxiv} applied to a simple non-coupled system of $N$ single equations.

\section{Proof of Theorem \ref{thm}}\label{sec3}

\subsection{Auxiliary results} 

The desired essential boundedness of the weak solutions to \eqref{P} will be obtained on the base of suitable decay estimates for the total mass of its $\alpha$-th component $u^\alpha$ on the super-level set of $u^\alpha,$ and these will be derived with the aid of two auxiliary results.

The first one is a trace inequality for the weighted Lebesgue norms, and we will use it in a form given by D.R.~Adams in \cite{Ad}, referring the reader also to the deep results of V.G.~Maz'ya \cite{Mz1,Mz2} for more general situations.
\begin{lem}\label{lem2}
Let $\m$ be a positive Radon measure with support  in $\Omega\subset\R^n,$ such that 
$$
\m(B_\rho(x))\leq A \rho^{a}\qquad
\forall\ x\in\R^n,\ \forall\ \rho>0,
$$ 
with a  constant $A,$ and where
$$
a=\frac{q}{t}(n-t),\quad 1<t<q<\infty,\quad t<n.
$$

Then
$$
\left(\int_\Omega |v(x)|^q\;d\m\right)^{1/q} \leq C(n,t,q) A^{1/q}
\left(\int_\Omega |Dv(x)|^t\;dx\right)^{1/t}
$$
for each $v\in W^{1,t}_0(\Omega).$
\end{lem}

The second tool in obtaining the estimate \eqref{EST} is a sort of maximal principle which goes back to Ph.~Hartman and G.~Stampacchia~\cite{HS} (see also \cite[Chapter~II, Lemma~5.1]{LU}).
\begin{lem}\label{lem3}  Let $\tau\colon\R\to[0,\infty)$ be a non-increasing function with the property that there exist constants $C>0,$ $k_0\geq0,$ $\nu>0$ and $\sigma\in[0,1+\nu]$ such that
$$
\int_k^\infty \tau(t)\;dt \leq C k^\sigma \big(\tau(k)\big)^{1+\nu}\quad \forall\ k\geq k_0.
$$
Then $\tau$ obeys the {\em finite time extinction property,} that is, there exists a number $k_{\max},$ depending on $C,$ $k_0,$ $\nu,$ $\sigma$ and $\int_{k_0}^\infty \tau(t)\;dt,$ such that
$$
\tau(k)=0\quad \forall\ k\geq k_{\max}.
$$
\end{lem}

\subsection{Higher gradient integrability} 
The regularity theory of nonlinear elliptic systems is more delicate matter in comparison with that of single equations, and the unique positive regularity result which holds true also for systems is the better integrability of the gradient in the spirit of  Gehring--Giaquinta--Modica. The weak solution of \eqref{P} obeys that property under fairly general controlled growth hypotheses \eqref{CG} and \textit{general} coercivity of the operator considered.
\begin{lem}\label{lem4}
Assume \eqref{A}, \eqref{CG} with
$\varphi\in L^r(\Omega),$ $r>\frac{p}{p-1};$ $
\psi \in L^s(\Omega),$ $s>\frac{np}{np+p-n},$ together with
\begin{equation}\label{CC}
\bA(x,\bz,\bxi)\cdot \bxi:=\sum_{\alpha=1}^N\sum_{i=1}^n a^\alpha_i(x,\bz,\bxi)\xi^\alpha_i
 \geq \varkappa |\bxi|^p-\Lambda |\bz|^{p^*} -\Lambda \varphi^{\frac{p}{p-1}}(x)
\end{equation}
for almost all $x\in\Omega,$ all $\bz\in \R^N$ and all $\bxi\in \R^{N\times n}.$

Let $\bu\in W^{1,p}_0(\Omega;\R^N)$ be a weak solution of the system \eqref{P}. Then there is a number $\varepsilon>0$ such that $D\bu\in L^q(\Omega;\R^{N\times n})$ for all $q\in(p,p+\varepsilon).$
\end{lem}
 
The proof of Lemma \ref{lem4} relies heavily on the \textit{reverse H\"older inequality} and repeats the lines of the proof of \cite[Theorem~2.2, Chapter~V]{G}, where its variant in the case $p=2$ is given (see also the discussion about the boundary higher integrability on pages 151--154 of \cite{G}).

\begin{rem}\label{rem-HI}
It is to be noted that the number $\varepsilon$ and the $L^q(\Omega;\R^{N\times n})$-bound for  the gradient $D\bu$ in Lemma~\ref{lem4} depend on  known quantities and also on the solution $\bu$ itself. More precisely, the dependence on $\bu$ relies on the uniform integrability of $|D\bu|^p$ over $\Omega$ in case $n\geq3,$ and on $\|D\bu\|_{L^p(\Omega;\R^{N\times n})}$ when $n=2.$

Moreover, in the particular case $p=n,$ Lemma~\ref{lem4} yields $\bu\in W^{1,n+\varepsilon}(\Omega;\R^N)$ whence $\bu$ is automatically globally bounded and H\"older continuous in $\ol\Omega$ by the Sobolev imbedding theorem.
\end{rem}

It is clear that, because of 
$$
|\bxi|^p\leq \left(\sum_{\alpha=1}^N |\bxi^\alpha|\right)^p\leq N^{p-1} \sum_{\alpha=1}^N |\bxi^\alpha|^p,
$$
\eqref{CWC} implies \eqref{CC} with new constants $\varkappa$ and $\Lambda$ depending on $N.$ 
Similarly, \eqref{CG} are automatically satisfied if \eqref{CWCG} hold. This means that 
 Lemma~\ref{lem4} applies to \eqref{P} under the hypotheses of Theorem~\ref{thm}, and combining it  with the Sobolev imbedding theorem, we get the existence of an exponent $p_0>p$ such that any weak solution
$\bu\in W^{1,p}_0(\Omega;\R^N)$ of \eqref{P} belongs to $W^{1,p_0}(\Omega;\R^N)\cap L^{p_0^*}(\Omega;\R^N)$ and
\begin{equation}\label{1}
\|\bu\|_{L^{p_0^*}(\Omega;\R^N)}+
\|D\bu\|_{L^{p_0}(\Omega;\R^{N\times n})}\leq C
\end{equation}
with $C$ depending on the quantities listed above.

\subsection{Decay estimates for the total mass of the solution}

Without loss of generality, we suppose that the functions $\varphi$ and $\psi$ are extended as zero outside $\Omega.$

\textit{Case 1: $p<n.$} Define the measure
$$
d\m= \left(\chi_\Omega(x)+\varphi^\frac{p}{p-1}(x)+\psi(x)+|\bu(x)|^{\frac{p^2}{n-p}}\right)dx
$$
with $\chi_\Omega$ being the characteristic function of $\Omega$  and the Lebesgue measure $dx.$

We have
$$
\int_{B_\rho} \chi_\Omega(x)\;dx=|B_\rho\cap\Omega|\leq (1-A_\Omega)\rho^n=(1-A_\Omega)\rho^{n-p+p}
$$
for any ball $B_\rho$ as consequence of \eqref{A}. Later on, using the hypotheses \eqref{M}, it is easy to check that
\begin{align*}
\int_{B_\rho} \varphi^\frac{p}{p-1}(x)\;dx \leq&\ 
\rho^{n-p+\left(p-\frac{p(n-\lambda)}{r(p-1)}\right)}
\|\varphi\|^\frac{p}{p-1}_{L^{r,\lambda}(\Omega)},\\
\int_{B_\rho} \psi(x)\;dx \leq&\ 
\rho^{n-p+\left(p-\frac{n-\mu}{s}\right)}
\|\psi\|_{L^{s,\mu}(\Omega)}
\end{align*}
where $p-\frac{p(n-\lambda)}{r(p-1)}>0$ and
$p-\frac{n-\mu}{s}>0$ as it follows from $(p-1)r+\lambda>n$ and $ps+\mu>n.$
Finally, the H\"older inequality and \eqref{1} imply
\begin{equation}\label{1'}
\int_{B\rho} |\bu(x)|^{\frac{p^2}{n-p}}\;dx \leq
\rho^{n-p+\left(p-\frac{np^2}{(n-p)p_0^*}\right)}
\|\bu\|_{L^{p_0^*}(\Omega;\R^N)}\leq 
 C \rho^{n-p+\left(p-\frac{np^2}{(n-p)p_0^*}\right)}
\end{equation}
with $p-\frac{np^2}{(n-p)p_0^*}>0$ since
$p_0^*>p^*=\frac{np}{n-p}.$

At this point, defining
$$
\delta:=\min\left\{p,p-\dfrac{p(n-\lambda)}{r(p-1)},p-\dfrac{n-\mu}{s},p-\dfrac{np^2}{(n-p)p_0^*} \right\}>0,
$$
it turns out that
\begin{equation}\label{2}
\m(B_\rho)\leq A\rho^{n-p+\delta},
\end{equation}
where $A$ is a constant depending on the data of \eqref{P} and on $\bu$ in the sense of Remark~\ref{rem-HI}.

We define now, for arbitrary $k\geq1$ and any
$\alpha\in\{1,\ldots,N\},$ the functions
\begin{align*}
W(x):=&\ \max\big\{|u^1(x)|,\ldots,|u^N(x)|\big\},\\
w(x):=&\ \max\big\{W(x)-k,0\big\},\\
w^\alpha(x):=&\ \max\big\{|u^\alpha(x)|-k,0\big\}
\end{align*}
and the corresponding super-level sets
\begin{align*}
\Omega_k:=&\ \big\{x\in\Omega\colon\ W(x)>k\big\}\equiv\big\{x\in\Omega\colon\ w(x)>0\big\},\\
\Omega_k^\alpha=&\ \big\{x\in\Omega\colon\ |u^\alpha(x)|>k\big\}\equiv\big\{x\in\Omega\colon\ w^\alpha(x)>0\big\}.
\end{align*}
We have that $u^\alpha\in W^{1,p}_0(\Omega)\cap W^{1,p_0}(\Omega)$ by \eqref{1} and therefore, since the Sobolev spaces are lattices, $W,w\in W^{1,p}_0(\Omega)\cap W^{1,p_0}(\Omega)$ as well.

Further on, $w\equiv0$ on $\Omega\sm\Omega_k,$ $w\in W^{1,p}_0(\Omega)$ and the H\"older inequality leads to
$$
\int_\Omega w(x)\;d\m=\int_{\Omega_k} w(x)\;d\m \leq
\left(\int_{\Omega_k} d\m\right)^{1-1/q} 
\left(\int_{\Omega_k} |w(x)|^q\;d\m\right)^{1/q}
$$
with suitable $q>1.$ Thanks to \eqref{2}, we can estimate the $L^q$-norm of $w$ by means of the  trace inequality, Lemma~\ref{lem2}.
In particular, we choose
$$
t=p,\quad q=\dfrac{p(n-p+\delta)}{n-p},\quad a=n-p+\delta
$$
there, in order to get
\begin{equation}\label{3}
\int_{\Omega_k} w(x)\;d\m \leq C \big(\m(\Omega_k)\big)^{1-\frac{n-p}{p(n-p+\delta)}}\left( \int_{\Omega_k} |Dw(x)|^p\;dx\right)^{1/p}.
\end{equation}

We will use the structure hypotheses \eqref{CWC}, \eqref{CWCG}, \eqref{SC} and \eqref{M} in order to estimate the $p$-energy of $w$ on the right-hand side of \eqref{3}, and the fact that $\bu\in W^{1,p}_0(\Omega;\R^N)$ is a weak solution of \eqref{P}. 

For this goal consider the function $
\bv=\big(v^1,\ldots,v^N\big)\colon\ \Omega\to\R^N$ with components 
$$
v^\alpha(x):=w^\alpha(x){.}\text{sign\,}u^\alpha(x).
$$
It is clear that $\bv\in W^{1,p}_0(\Omega;\R^N)$ and $Dv^\alpha=Dw^\alpha{.}\text{sign\,}u^\alpha=Du^\alpha$ for a.a. $x\in\Omega_k^\alpha.$ Using $\bv$ as a test function in \eqref{WS} and taking into account that $v^\alpha\equiv0$ on $\Omega\sm\Omega_k^\alpha,$ we get
\begin{align}\label{4}
&\sum_{\alpha=1}^N\int_{\Omega_k^\alpha}
\sum_{i=1}^n
a_i^\alpha\big(x,\bu(x),D\bu(x)\big) D_iv^\alpha(x)\;dx\\
\nonumber
&\qquad\qquad
+\sum_{\alpha=1}^N\int_{\Omega_k^\alpha}  b^\alpha\big(x,\bu(x),D\bu(x)\big)v^\alpha(x)\;dx=0.
\end{align}

The first integral will be estimated from below with the aid of the componentwise coercivity condition \eqref{CWC} and the fact that
$Dv^\alpha(x)=Du^\alpha(x)$ a.e. on $\Omega_k^\alpha:$ 
\begin{align}\label{5}
&\sum_{\alpha=1}^N\int_{\Omega_k^\alpha}
\sum_{i=1}^n
a_i^\alpha\big(x,\bu(x),D\bu(x)\big) D_iv^\alpha(x)\;dx\\
\nonumber
&\qquad\qquad
=\sum_{\alpha=1}^N\int_{\Omega_k^\alpha}
\sum_{i=1}^n
a_i^\alpha\big(x,\bu(x),D\bu(x)\big) D_iu^\alpha(x)\;dx\\
\nonumber
&\qquad\qquad
\geq \varkappa \sum_{\alpha=1}^N\int_{\Omega_k^\alpha}  |Du^\alpha(x)|^p\;dx\\
\nonumber
&\qquad\qquad\quad
 -\Lambda N\left(
\int_{\Omega_k^\alpha}  |\bu(x)|^{p^*}\;dx +
\int_{\Omega_k^\alpha}  \varphi^\frac{p}{p-1}(x)\;dx\right).
\end{align}

To estimate the second integral in \eqref{4} from below in case when \eqref{CWCG} is satisfied, we will bound first the absolute value of the integrand. Thus, using 
\begin{equation}\label{5'}
0<\dfrac{|u^\alpha(x)|-k}{|u^\alpha(x)|}<1\quad
\text{a.e.\ in}\ \Omega_k^\alpha,
\end{equation}
\eqref{CWCG} and the Young inequality, we get
\begin{align*}
\big|b^\alpha(x,\bu(x),&D\bu(x))v^\alpha(x)\big|
= \big|b^\alpha(x,\bu(x),D\bu(x))\big|w^\alpha(x)\\
=&\ \big|b^\alpha(x,\bu(x),D\bu(x))\big|{.}|u^\alpha(x)|\dfrac{|u^\alpha(x)|-k}{|u^\alpha(x)|}\\
\leq&\ \Lambda|u^\alpha(x)|\left(
\psi(x)+|\bu(x)|^{p^*-1}+
|Du^\alpha(x)|^{\frac{p(p^*-1)}{p^*}}\right)\\
\leq&\ \Lambda\left(
\psi(x)|u^\alpha(x)|+C(\varepsilon)|\bu(x)|^{p^*}+\varepsilon
|Du^\alpha(x)|^{p}\right)
\end{align*}
for a.a. $x\in \Omega_k^\alpha$ and with arbitrary $\varepsilon>0$ to be chosen later.

This way,
\begin{align}\label{5''}
&\sum_{\alpha=1}^N\int_{\Omega_k^\alpha} 
b^\alpha(x,\bu(x),D\bu(x))v^\alpha(x)\;dx\\
\nonumber
& \qquad\geq
-\Lambda \sum_{\alpha=1}^N\int_{\Omega_k^\alpha}
\psi(x)|u^\alpha(x)|\;dx
-\Lambda N C(\varepsilon)
\int_{\Omega_k^\alpha}|\bu(x)|^{p^*}dx\\
\nonumber
&\qquad\quad
-\Lambda\varepsilon\sum_{\alpha=1}^N
\int_{\Omega_k^\alpha}
|Du^\alpha(x)|^{p}\;dx.
\end{align}
Making use of \eqref{5'} and the Young inequality, the same bound follows also if, instead of \eqref{CWCG},  
the controlled growth assumptions \eqref{CG}
are verified together with the sign condition \eqref{SC}.

We employ now \eqref{5} and \eqref{5''} into \eqref{4}, and a suitable choice of $\varepsilon>0$ yields the basic energy estimate
\begin{align}\label{6}
\sum_{\alpha=1}^N\int_{\Omega_k^\alpha}
|Du^\alpha(x)|^{p}\;dx\leq&\ C\left(
\int_{\Omega_k^\alpha}  \varphi^\frac{p}{p-1}(x)\;dx
+\sum_{\alpha=1}^N\int_{\Omega_k^\alpha} \psi(x)|u^\alpha(x)|\;dx\right.\\
\nonumber
&\qquad \left.
+ \int_{\Omega_k^\alpha}|\bu(x)|^{p^*}\;dx
\right).
\end{align}

Our aim now is to estimate the right-hand side of \eqref{6} in terms of suitable power of the level $k$ multiplied by the $\m$-measure of the level set $\Omega_k.$ For, the definition of the measure $\m$ gives immediately
\begin{equation}\label{7}
\int_{\Omega_k^\alpha}  \varphi^\frac{p}{p-1}(x)\;dx\leq \int_{\Omega_k}  \varphi^\frac{p}{p-1}(x)\;dx \leq \m(\Omega_k)
\end{equation}
since $\Omega_k^\alpha\subset \Omega_k.$

We have further
\begin{align}\label{8}
\int_{\Omega_k^\alpha} \psi(x)|u^\alpha(x)|\;dx =&\ \int_{\Omega_k^\alpha} \psi(x)\big(|u^\alpha(x)|-k+k\big)\;dx\\
\nonumber
\leq &\ \int_{\Omega_k^\alpha} \psi(x)|w^\alpha(x)|\;dx + k\int_{\Omega_k^\alpha} \psi(x)\;dx\\
\nonumber
\leq &\ \int_{\Omega_k^\alpha} \psi(x)|v^\alpha(x)|\;dx + k\int_{\Omega_k} \psi(x)\;dx\\
\nonumber
\leq &\ \int_{\Omega_k^\alpha} \psi(x)|v^\alpha(x)|\;dx + k\m(\Omega_k). 
\end{align}
Setting $d\m':=\psi(x)\;dx,$ it follows from $\psi\in L^{s,\mu}(\Omega)$ that 
$$
\m'(B_\rho)\leq \|\psi\|_{L^{s,\mu}(\Omega)}
\rho^{n-\frac{n-\mu}{s}}
$$
for any ball $B_\rho.$ Therefore, Lemma~\ref{lem2} applied with $t=p$ and $q=\frac{p\left(n-\frac{n-\mu}{s}\right)}{n-p}$
(recall $ps+\mu>n$ by \eqref{M}), $Dv^\alpha=Du^\alpha$ on $\Omega_k^\alpha,$ and Young's inequality yield
\begin{align*}
\int_{\Omega_k^\alpha} \psi(x)|v^\alpha(x)|\;dx =&\ \int_{\Omega_k^\alpha} |v^\alpha(x)|\;d\m'\leq \left( \int_{\Omega_k^\alpha} d\m'\right)^{1-\frac{1}{q}}\left( \int_{\Omega_k^\alpha} |v^\alpha(x)|^q\;d\m'\right)^\frac{1}{q}\\
\leq &\ C\big(\m'(\Omega_k^\alpha)\big)^{1-\frac{1}{q}}
\left( \int_{\Omega_k^\alpha} |Dv^\alpha(x)|^p\;dx\right)^\frac{1}{p}\\
\leq &\ C\big(\m'(\Omega_k)\big)^{1-\frac{1}{q}}
\left( \int_{\Omega_k^\alpha} |Du^\alpha(x)|^p\;dx\right)^\frac{1}{p}\\
\leq &\ \varepsilon \int_{\Omega_k^\alpha} |Du^\alpha(x)|^p\;dx +C(\varepsilon)
\big(\m'(\Omega_k)\big)^{\frac{q-1}{q}\frac{p}{p-1}}
\end{align*}
with arbitrary $\varepsilon>0.$ Further on,
$$
\big(\m'(\Omega_k)\big)^{\frac{q-1}{q}\frac{p}{p-1}}\leq 
\big(\m(\Omega_k)\big)^{\frac{q-1}{q}\frac{p}{p-1}}\leq
\big(\m(\Omega)\big)^{\frac{q-1}{q}\frac{p}{p-1}-1} \m(\Omega_k),
$$
while
$$
\m(\Omega)\leq |\Omega|+C\left(\|\varphi\|^\frac{p}{p-1}_{L^{r,\lambda}(\Omega)} +\|\psi\|_{L^{s,\mu}(\Omega)}+\|\bu\|^\frac{p^2}{n-p}_{L^{p_0^*}(\Omega;\R^N)}\right)
$$
whence $\m(\Omega)$ is bounded in terms of known quantities and $\|D\bu\|_{L^p(\Omega;\R^{N\times n})}$ as follows from \eqref{1} and Remark~\ref{rem-HI}. Summarizing, \eqref{8} becomes
\begin{equation}\label{9}
\int_{\Omega_k^\alpha} \psi(x)|u^\alpha(x)|\;dx 
\leq \varepsilon \int_{\Omega_k^\alpha} |Du^\alpha(x)|^p\;dx + C(\varepsilon) k \m(\Omega_k)
\end{equation}
with arbitrary $\varepsilon>0$ to be chosen later.

Regarding the third term in \eqref{6}, 
we use $\Omega_k^\alpha\subset\Omega_k,$ $|\bu(x)|\leq N^{p/2}W(x)$ and $W(x)=w(x)+k$ a.e. on $\Omega_k,$ in order to conclude that
\begin{align*}
|\bu(x)|^{p^*}=&\ |\bu(x)|^{p} |\bu(x)|^{p^*-p} \\
\leq &\ C(N,p) |W(x)|^{p} |\bu(x)|^{\frac{p^2}{n-p}}\\
= &\ C(N,p) \big|W(x)-k+k\big|^{p} |\bu(x)|^{\frac{p^2}{n-p}}\\
\leq &\ C(N,p)2^{p-1}\left( |w(x)|^{p} |\bu(x)|^{\frac{p^2}{n-p}}+k^p |\bu(x)|^{\frac{p^2}{n-p}}\right)
\end{align*}
for a.a. $x\in\Omega_k.$ Therefore
\begin{align*}
\int_{\Omega_k^\alpha} |\bu(x)|^{p^*}dx\leq&\
	\int_{\Omega_k} |\bu(x)|^{p^*}dx\\
\leq &\ C\left(\int_{\Omega_k} |w(x)|^{p} |\bu(x)|^{\frac{p^2}{n-p}}\;dx+k^p \m(\Omega_k)\right)
\end{align*}
and the integral on the right will be estimated once again with the help of Lemma~\ref{lem2}.

For, remembering \eqref{1'} and having in mind
$$
n-p+\left(p-\dfrac{np^2}{(n-p)p_0^*}\right)>n-p,
$$
we pick a number $t<p$ such that
$$
n-p+\left(p-\dfrac{np^2}{(n-p)p_0^*}\right)>\dfrac{p}{t}(n-t)>n-p,
$$
and then \eqref{1'} yields
$$
\int_{B_\rho} |\bu(x)|^\frac{p^2}{n-p}\;dx \leq C\left(n,p,\mathrm{diam}\,\Omega,\|\bu\|_{L^{p_0^*}(\Omega;\R^N)}\right) \rho^{\frac{p}{t}(n-t)}.
$$
Defining the measure $d\m'':=|\bu(x)|^\frac{p^2}{n-p}\;dx$ and taking $q=p$ in Lemma~\ref{lem2}, we get
$$
\int_{\Omega_k} |w(x)|^p|\bu(x)|^\frac{p^2}{n-p}\;dx=\int_{\Omega_k} |w(x)|^p\;d\m''\leq C\left(
\int_{\Omega_k} |Dw(x)|^t\;dx
\right)^{\frac{p}{t}}
$$
since $w\in W^{1,p}_0(\Omega_k),$
and the H\"older inequality gives
$$
\int_{\Omega_k} |w(x)|^p|\bu(x)|^\frac{p^2}{n-p}\;dx\leq C |\Omega_k|^{\frac{p}{t}-1}
\int_{\Omega_k} |Dw(x)|^p\;dx.
$$
This way, the last term in \eqref{6} estimates as
\begin{equation}\label{10}
\int_{\Omega_k^\alpha}|\bu(x)|^{p^*}dx\leq C\left(k^p\m(\Omega_k)+
 |\Omega_k|^{\frac{p}{t}-1}
\int_{\Omega_k} |Dw(x)|^p\;dx
\right).
\end{equation}

Employing \eqref{7}, \eqref{9} and \eqref{10} into \eqref{6}, remembering $k\geq1,$ and choosing $\varepsilon$ small enough in \eqref{9}, we obtain
\begin{equation}\label{11a}
\sum_{\alpha=1}^N
\int_{\Omega_k^\alpha} |Du^\alpha(x)|^p\;dx
\leq C\left(|\Omega_k|^{\frac{p}{t}-1}
\int_{\Omega_k} |Dw(x)|^p\;dx
+k^p\m(\Omega_k)\right).
\end{equation}
Before proceed further, we note that the term on the left-hand side can be estimated from below by $\int_{\Omega_k} |Dw(x)|^p\;dx.$ In fact, $W(x)=|u^\alpha(x)|$ for some $\alpha$ implies
$|DW(x)|=|Du^\alpha(x)|.$ Moreover, if $x\in\Omega_k$ and $W(x)=|u^\alpha(x)|$ then $|u^\alpha(x)|>k$ whence $x\in\Omega_k^\alpha.$ If $x\in\Omega_k^\alpha$ for some $\alpha$ then $|u^\alpha(x)|>k$ and thus $W(x)>k$ as well, that means $x\in\Omega_k.$ Therefore $\Omega_k=\cup_{\alpha=1}^N \Omega_k^\alpha$ and
\begin{align*}
\sum_{\alpha=1}^N
\int_{\Omega_k^\alpha} |Du^\alpha(x)|^p\;dx \geq&\ \sum_{\alpha=1}^N
\int_{\{x\in\Omega_k\colon\ W(x)=|u^\alpha(x)|\}} |Du^\alpha(x)|^p\;dx\\
\geq &\ \int_{\Omega_k} |DW(x)|^p\;dx.
\end{align*}
At this point we use $DW(x)=Dw(x)$ on $\Omega_k$ since $w=W-k$ there, in order to rewrite \eqref{11a} as
\begin{equation}\label{11}
\int_{\Omega_k} |Dw(x)|^p\;dx
\leq C\left(|\Omega_k|^{\frac{p}{t}-1}
\int_{\Omega_k} |Dw(x)|^p\;dx
+k^p\m(\Omega_k)\right).
\end{equation}
Our next aim is to move the first term above on the left and this could be done if $|\Omega_k|$ is small enough for large $k.$ This is just the case because of
$$
k^{p^*}|\Omega_k^\alpha|\leq \int_{\Omega_k^\alpha} |u^\alpha(x)|^{p^*}\;dx \leq \int_{\Omega} |\bu(x)|^{p^*}\;dx\leq C\|D\bu\|_{L^p(\Omega;\R^{N\times n})}^{p^*}
$$
and
$$
|\Omega_k|=\left| \bigcup_{\alpha=1}^N \Omega_k^\alpha\right|\leq \sum_{\alpha=1}^N |\Omega_k^\alpha|,
$$
that means there exists a $k_0,$ depending on known quantities and on $\|D\bu\|_{L^p(\Omega;\R^{N\times n})},$ such that if $k\geq k_0$ then
$|\Omega_k|$ can be done small enough to ensure $C|\Omega_k|^{\frac{p}{t}-1}$ in \eqref{11} is less than $1/2.$
Then \eqref{11} becomes
\begin{equation}\label{12}
\int_{\Omega_k} |Dw(x)|^{p}\;dx
\leq Ck^p\m(\Omega_k)\quad \forall\ k\geq k_0,
\end{equation}
and \eqref{3} rewrites into
\begin{equation}\label{13}
\int_{\Omega_k} w(x)\;d\m
\leq Ck\big(\m(\Omega_k)\big)^{1+\frac{\delta}{p(n-p+\delta)}}\quad \forall\ k\geq k_0,\quad \delta>0.
\end{equation}

The Cavalieri principle gives
$$
\int_{\Omega_k} w(x)\;d\m =\int_{\Omega_k} \big(W(x)-k\big)\;d\m=
\int_k^\infty \m(\Omega_t)\;dt
$$
with $\tau(t):=\m(\Omega_t),$ 
and \eqref{13} takes on the form
$$
\int_k^\infty \tau(t)\;dt\leq Ck\big(\tau(k))^{1+\nu}\quad \forall\ k\geq k_0,\ \nu=\dfrac{\delta}{p(n-p+\delta)}>0.
$$

It follows now by the Hartman--Stampacchia maximum principle (Lemma~\ref{lem3}) that there exists a number $k_{\max}$ such that
$$
\tau(k)=\m\big(\{x\in\Omega\colon\ W(x)>k\}\big)=0\quad \forall\ k\geq k_{\max},
$$
that is,
$$
W(x)\leq k_{\max}\quad \text{for a.a.}\ x\in\Omega
$$
which yields the desired bound \eqref{EST} in case $p<n.$

\bigskip

\textit{Case 2: $p=n.$} The claim \eqref{EST} follows immediately on the base of Lemma~\ref{lem4}, \eqref{A} and the Sobolev imbedding theorem. In fact, Lemma~\ref{lem4} yields 
$\bu\in W^{1,n+\varepsilon}(\Omega)$ for some $\varepsilon>0$ whence $\bu$ is even H\"older continuous in $\ol\Omega$ as consequence of \eqref{A} and the Morrey lemma. 

However, the result of Theorem~\ref{thm} holds true also in less regular domains with \textit{$p$-thick complements} (see \cite{BPS-arxiv}) when, in general, the measure density condition \eqref{A} fails, and for the sake of future references we will sketch the proof of \eqref{EST} also in the case $p=n.$
We will argue as above, taking into account that $p^*>n$ could be any number. Thus, we choose a $p'<n$ and close enough to $n,$ in a way that
$$
n^*=p^*=\dfrac{n^2}{(n-p')(n+1)}.
$$
Having in mind that ${p'}^*=\frac{np'}{n-p'},$ it is clear that
$$
n^*=p^*<{p'}^*\quad \text{and}\quad
\dfrac{n(n^*-1)}{n^*}=\dfrac{p'({p'}^*-1)}{{p'}^*}.
$$
Since the componentwise conditions \eqref{CWC}, \eqref{CWCG} and \eqref{SC} are essential for large values of $|\bz|$ and $|\bxi^\alpha|,$ assuming without loss of generality that $\varphi(x)\geq 1,$ we have for each $\alpha\in\{1,\ldots,N\}$ and all $|\bz|,|\bxi^\alpha|\geq 1$ that
\begin{align*}
|\bxi^\alpha|^n \geq&\ |\bxi^\alpha|^{p'},\\
|\bxi^\alpha|^\frac{n(n^*-1)}{n^*} =&\ |\bxi^\alpha|^\frac{{p'}({p'}^*-1)}{{p'}^*},\\
|\bz|^{n^*} \leq&\ |\bz|^{{p'}^*},\\
\varphi^\frac{n}{n-1}\leq&\ \varphi^\frac{p'}{p'-1}.
\end{align*}

Therefore, \eqref{CWC}, \eqref{CWCG} and \eqref{SC} take one the forms
\begin{align}\label{14}
\sum_{i=1}^n a^\alpha_i(x,\bz,\bxi)\xi^\alpha_i
 \geq&\ \varkappa |\bxi^\alpha|^{p'}-\Lambda |\bz|^{{p'}^*} -\Lambda \varphi^{\frac{{p'}}{{p'}-1}}(x),\\
\label{15}
\big|b^\alpha(x,\bz,\bxi)\big|\leq &\  \Lambda\Big(\psi(x)+|\bz|^{{p'}^*-1} +|\bxi^\alpha|^{\frac{{p'}({p'}^*-1)}{{p'}^*}}\Big),\\
\label{16}
b^\alpha(x,\bz,\bxi){.}\mathrm{sign\,}z^\alpha \geq &\ 
- \Lambda\Big(\psi(x)+|\bz|^{{p'}^*-1} +|\bxi^\alpha|^{\frac{{p'}({p'}^*-1)}{{p'}^*}}\Big),
\end{align}
respectively.

Define the measure
$$
d\overline{\m}:=
\left(\chi_\Omega(x)+\varphi^\frac{p'}{p'-1}(x)+\psi(x)+|\bu(x)|^{\frac{{p'}^2}{n-p'}}\right)dx
$$
for which
$$
\overline{\m}(B_\rho)\leq K \rho^{n-p'+\delta}
$$
with a positive $\delta,$ after increasing $p'<n,$ if necessary, to have
$$
r>\dfrac{p'}{p'-1},\quad p'>\dfrac{n-\lambda}{r}+1,\quad p'>\dfrac{n-\mu}{s}.
$$
With the function $w(x)$ and the set $\Omega_k,$ defined as in the previous case, it is clear that
$$
\int_{\{x\in\Omega_k\colon\ |Dw(x)|<1\}} |Dw(x)|^{p'}\;dx \leq |\Omega_k|\leq k^{p'}\overline{\m}(\Omega_k),\quad k\geq 1,
$$
whereas, using \eqref{14}, \eqref{15} and \eqref{16}, the energy
$$
\int_{\{x\in\Omega_k\colon\ |Dw(x)|\geq 1\}} |Dw(x)|^{p'}\;dx
$$
is estimated in the same manner as before, yielding finally
$$
\int_{\Omega_k} |Dw(x)|^{p'}\;dx
\leq C k^{p'} \overline{\m}(\Omega_k)\quad \forall\ k\geq k_0
$$
(cf. \eqref{12}). 

The desired bound \eqref{EST} follows as in \textit{Case 1} through the Hartman--Stampacchia maximum principle, and this completes the
proof of Theorem~\ref{thm}.\hfill $\qed$

\section{Morrey regularity of the gradient}\label{sec4}

As already mentioned before, the weak solution of \eqref{P} is globally H\"older continuous in $\Omega$ when $p=n$ since $D\bu\in L^{n+\varepsilon}(\Omega;\R^{N\times n})$ as consequence of the higher gradient integrability result, Lemma~\ref{lem4}. Invoking the inclusion properties of the Morrey spaces (\cite{Pic}), it is not hard to see that the gradient $D\bu$ of the weak solution possesses some Morrey regularity, namely $D\bu\in L^{n,n-\frac{n^2}{n+\varepsilon}}(\Omega;\R^{N\times n}).$

It turns out that, thanks to the hypotheses \eqref{M}, this property is own by the gradient of the bounded solution also when $p<n$ under the general controlled growth \eqref{CG} and
coercivity \eqref{CC} requirements on the system, without any componentwise restrictions.
\begin{thm}\label{thm2}
Assume $p\in (1,n),$ \eqref{A}, \eqref{CG}, \eqref{M} and \eqref{CC}, and let $\bu \in  L^{\infty}(\Omega;\R^N)\cap W^{1,p}_0(\Omega;\R^N)$ be a {\em bounded} weak solution of \eqref{P}.

Then $D\bu\in L^{n,n-p}(\Omega;\R^{N\times n}).$
 Precisely, 
\begin{equation}\label{MORREY}
\int_{B_\rho(x_0)\cap\Omega} |D\bu(x)|^p\;dx
\leq C \rho^{n-p}\quad \forall x_0\in \ol\Omega,\ \forall \rho\in(0,\mathrm{diam}\,\Omega)
\end{equation}
with a constant $C$ depending on known quantities and on $M=\|\bu\|_{L^{\infty}(\Omega;\R^N)}.$
\end{thm}
\begin{proof}
Fix a point $x_0\in\Omega$ and consider a positive cut-off function $\zeta\in C^\infty_0(B_{2\rho}(x_0))$ such that $\zeta\equiv1$ on $B_\rho(x_0)$ and $|D\zeta|\leq 1/\rho.$

Suppose $B_{2\rho}\Subset\Omega$ and consider the  function $\bv\in W^{1,p}_0(\Omega;\R^N)$ with components $v^\alpha(x)=\e^{u^\alpha(x)}\zeta^p(x).$ Employing $\bv$ as a test function in \eqref{WS} gives
\begin{align}\label{N1}
&\sum_{\alpha=1}^N\sum_{i=1}^n
\int_{\Omega}
a_i^\alpha\big(x,\bu(x),D\bu(x)\big)\e^{u^\alpha(x)}\zeta^p(x) D_iu^\alpha(x)\;dx\\
\nonumber
&\qquad\qquad
+p\sum_{\alpha=1}^N\sum_{i=1}^n
\int_{\Omega}
a_i^\alpha\big(x,\bu(x),D\bu(x)\big)\e^{u^\alpha(x)}\zeta^{p-1}(x) D_i\zeta(x)\;dx\\
\nonumber
&\qquad\qquad
+ \sum_{\alpha=1}^N
\int_{\Omega}
b^\alpha\big(x,\bu(x),D\bu(x)\big)\e^{u^\alpha(x)}\zeta^p(x) \;dx=0.
\end{align}

We will estimate the three integrands above, keeping in mind $|u^\alpha(x)|\leq M$ for a.a. $x\in\Omega.$ Thus, the coercivity condition \eqref{CC} yields
\begin{align*}
&\sum_{\alpha=1}^N\sum_{i=1}^n
a_i^\alpha\big(x,\bu(x),D\bu(x)\big)\e^{u^\alpha(x)}\zeta^p(x) D_iu^\alpha(x)\\
&\qquad\qquad \geq \e^{-M}\zeta^p(x) \left(
\varkappa |D\bu(x)|^p-\Lambda M^{p^*}-\Lambda \varphi^{\frac{p}{p-1}}(x)\right)\\
&\qquad\qquad \geq C\zeta^p(x) \left(
 |D\bu(x)|^p-1- \varphi^{\frac{p}{p-1}}(x)\right)\quad \text{for a.a.}\ x\in\Omega.
\end{align*}
Further on, the controlled growth assumptions $\eqref{CG}$ and the Young inequality imply
\begin{align*}
&\big|a_i^\alpha\big(x,\bu(x),D\bu(x)\big)\e^{u^\alpha(x)}\zeta^{p-1}(x) D_i\zeta(x)\big|\\
&\qquad\qquad
\leq \e^M \zeta^{p-1}(x)|D\zeta(x)|\left(\varphi(x)+M^{\frac{p^*(p-1)}{p}}+|D\bu(x)|^{p-1}\right)\\
&\qquad\qquad \leq 
\varepsilon \zeta^p(x)|D\bu(x)|^p +\varepsilon\zeta^p(x)\left(1+\varphi^{\frac{p}{p-1}}(x)\right)+C(\varepsilon)|D\zeta(x)|^p\quad \text{for a.a.}\ x\in\Omega
\end{align*}
and
\begin{align*}
&\big|b^\alpha\big(x,\bu(x),D\bu(x)\big)\e^{u^\alpha(x)}\zeta^{p}(x)\big|\\
&\qquad\qquad
\leq \e^M \zeta^{p}(x)\left(\psi(x)+M^{p^*-1}+|D\bu(x)|^{\frac{p(p^*-1)}{p^*}}\right)\\
&\qquad\qquad \leq 
\varepsilon \zeta^p(x)|D\bu(x)|^p +C(\varepsilon)\zeta^p(x)\big(1+\psi(x)\big)\quad \text{for a.a.}\ x\in\Omega.
\end{align*}
Employing these bounds in \eqref{N1}, choosing appropriately $\varepsilon>0$ and taking into account that $\zeta(x)=1$ on $B_\rho,$ we get
$$
\int_{B_\rho} |D\bu(x)|^p\;dx \leq \left(
\int_{B_{2\rho}} |D\zeta(x)|^p\;dx+
\int_{B_{2\rho}} \left(1+\varphi^{\frac{p}{p-1}}(x)+\psi(x)\right)\;dx\right).
$$
It is clear that
$$
\int_{B_{2\rho}} |D\zeta(x)|^p\;dx\leq C \rho^{n-p}
$$
in view of $|D\zeta|\leq 1/\rho,$ while 
$$
\int_{B_{2\rho}} \left(1+\varphi^{\frac{p}{p-1}}(x)+\psi(x)\right)\;dx\leq C\left(
\rho^n+\rho^{n-\frac{p(n-\lambda)}{r(p-1)}}+\rho^{n-\frac{n-\mu}{s}}\right)
$$
because of \eqref{M}. Furthermore, the requirements $(p-1)r+\lambda>n$ and $ps+\mu>n$ ensure
$$
n-\frac{p(n-\lambda)}{r(p-1)}>n-p\quad \text{and}\quad n-\frac{n-\mu}{s}>n-p,
$$
whence
$$
\int_{B_\rho} |D\bu(x)|^p\;dx
\leq C \rho^{n-p}.
$$
In case $B_{2\rho}\not\Subset\Omega,$ we use in \eqref{WS} a test function with components
$$
v^\alpha(x)=\left(\e^{|u^\alpha(x)|}-1\right)\zeta(x)\chi_\Omega(x)\,\text{sign}\,u^\alpha(x),
$$
and the measure density condition \eqref{A} together with similar arguments as above lead to the claim \eqref{MORREY}.
\end{proof}


\begin{thebibliography}{29}
\bibitem{Ad} 
\textsc{D.R. Adams,} 
{Traces of potentials arising from translation invariant operators,}
\textit{Ann. Scuola Norm. Sup. Pisa (3)} \textbf{25} (1971), 203--217.

\bibitem{Bjorn} 
\textsc{J. Bj\"orn,} 
{Boundedness and differentiability for nonlinear elliptic systems,}
\textit{Trans. Amer. Math. Soc.} \textbf{353} (2001), no. 11, 4545--4565.

\bibitem{BP-IUMJ} 
\textsc{S.-S. Byun, D.K. Palagachev,} 
{Boundedness of the weak solutions to quasilinear elliptic equations with Morrey data,}
\textit{Indiana Univ. Math. J.} \textbf{62} (2013), no. 5, 1565--1585.


\bibitem{BP-CalcVar} 
\textsc{S.-S. Byun, D.K. Palagachev,} 
{Morrey regularity of solutions to quasilinear elliptic equations over Reifenberg flat domains,} 
\textit{Calc. Var. Partial Differential Equations} \textbf{49} (2014), no. 1-2, 37--76.

\bibitem{BPS} 
\textsc{S.-S. Byun, D.K. Palagachev, P. Shin,} 
{Global continuity of solutions to quasilinear equations with Morrey data,} 
\textit{C. R. Math. Acad. Sci. Paris} \textbf{353} (2015), no. 8, 717--721.


\bibitem{BPS-arxiv} 
\textsc{S.-S. Byun, D.K. Palagachev, P. Shin,} 
{Global H\"older continuity of solutions to quasilinear equations with Morrey data,} (2015); 
\href{http://arxiv.org/abs/1501.06192}{arXiv:1501.06192}


\bibitem{DG1} 
\textsc{E. De Giorgi,}
{Sulla differenziabilit\`a e l'analiticit\`a delle estremali degli integrali multipli regolari,}
\textit{Mem. Accad. Sci. Torino. Cl. Sci. Fis. Mat. Nat. (3)} \textbf{3} (1957), 25--43.

\bibitem{DG2} 
\textsc{E. De Giorgi,}
{Un esempio di estremali discontinue per un problema variazionale di tipo ellittico,}
\textit{Boll. Un. Mat. Ital. (4)} \textbf{1} (1968), 135--137.


\bibitem{G} 
\textsc{M. Giaquinta,}
\textit{Multiple Integrals in the Calculus of Variations and Nonlinear Elliptic Systems,} Annals of Mathematics Studies, 105, Princeton University Press, Princeton, NJ, 1983.

\bibitem{HS} 
\textsc{Ph. Hartman, G. Stampacchia,} 
{On some non-linear elliptic differential-functional equations,}
\textit{Acta Math.} \textbf{115} (1966), 271--310. 


\bibitem{LU} 
\textsc{O.A. Ladyzhenskaya, N.N. Ural'tseva,}
\textit{Linear and Quasilinear Equations of Elliptic Type,} 2nd Edition revised, Nauka, Moscow,  1973, (in Russian).

\bibitem{Lan} 
\textsc{R. Landes,}
{Some remarks on bounded and unbounded weak solutions of elliptic systems,} 
\textit{Manuscripta Math.}  \textbf{64} (1989), no. 2, 227--234.

\bibitem{Leo} 
\textsc{F. Leonetti, P.-V. Petricca,}
{Regularity for solutions to some nonlinear elliptic systems,} 
\textit{Complex Var. Elliptic Equ.}  \textbf{56} (2011), no. 12, 1099--1113. 

\bibitem{Mz1} 
\textsc{V.G. Maz'ya,}
{Certain integral inequalities for functions of several variables (Russian),} 
\textit{Probl. Mat. Anal.}  \textbf{3} (1972), 33--68; English transl.: \textit{J. Sov. Math.} \textbf{1} (1973), 205--234.

\bibitem{Mz2} 
\textsc{V.G. Maz'ya,}
{Strong capacity-estimates for ``fractional'' norms (Russian),} 
\textit{Zap. Nau{\v{c}}n. Sem. Leningrad. Otdel. Mat. Inst. Steklov. (LOMI)}  \textbf{70} (1977), 161--168.

\bibitem{M} 
\textsc{M. Meier,}
{Boundedness and integrability properties of weak solutions of quasilinear elliptic systems,} 
\textit{J. Reine Angew. Math.}  \textbf{333} (1982), 191--220. 

\bibitem{N} 
\textsc{J. Nash,}
{Continuity of solutions of parabolic and elliptic equations,}
\textit {Amer. J. Math.}  \textbf{80} (1958),  931--954.

\bibitem{NS} 
\textsc{J. Ne\v{c}as, J. Star\'a,}
{Principio di massimo per i sistemi ellittici quasi-lineari non diagonali,}
\textit {Boll. Un. Mat. Ital. (4)}  \textbf{6} (1972),  1--10.

\bibitem{NA} 
\textsc{D.K. Palagachev, L.G. Softova,} 
{The Calder\'on--Zygmund property for quasilinear divergence form equations over Reifenberg flat domains,}
\textit{Nonlinear Anal.} \textbf{74} (2011), no. 5, 1721--1730.


\bibitem{Pic} 
\textsc{L.C. Piccinini,} 
{Inclusioni tra spazi di Morrey,}
\textit{Boll. Un. Mat. Ital. (4)} \textbf{2} (1969), 95--99. 

\bibitem{Sf} 
\textsc{L.G. Softova,}  
$L^p$-integrability of the gradient of solutions to quasilinear systems with discontinuous coefficients, 
\textit{Differential Integral Equations} \textbf{26} (2013), no. 9--10, 1091--1104.

\bibitem{Sof} 
\textsc{L.G. Softova,}
{Boundedness of the solutions to non-linear
systems with Morrey data,} 
\textit{Complex Var. Elliptic Equ.}  \textbf{63} (2018), \href{http://dx.doi.org/doi:10.1080/17476933.2017.1397642}{doi:10.1080/17476933.2017.1397642}

\bibitem{Sof2} 
\textsc{L.Softova,}
Maximum principle for a kind of elliptic systems with Morrey data,
In: S. Pinelas et al. (eds.), Different.  Difference Equ. Appli.,
Springer Proceedings in Mathematics \& Statistics \textbf{230}, 2018,  429--439.
\end{thebibliography}
\end{document}